\documentclass[11pt]{amsart}
\usepackage{color}
\usepackage{latexsym,amssymb,amsmath,youngtab}
\usepackage{epsfig}
\usepackage{amsmath,amsthm,amssymb,amscd}
\usepackage[mathscr]{eucal}
\usepackage{verbatim}

\newtheoremstyle{custom}% name
  {3pt}%      Space above
  {3pt}%      Space below
  {\slshape}%         Body font
  {}%         Indent amount (empty = no indent, \parindent = para indent)
  {\bfseries}% Thm head font
  {.}%        Punctuation after thm head
  { }%     Space after thm head: " " = normal interword space;
   {}%         Thm head spec (can be left empty, meaning `normal')
\theoremstyle{custom}
\newtheorem{theorem}{Theorem}[subsection]

\newtheorem{proposition}[theorem]{Proposition}
\newtheorem{proposition/definition}[theorem]{Proposition/Definition}

\theoremstyle{definition}
\newtheorem{definition}[theorem]{Definition}

\newtheorem{example}[theorem]{Example}

\theoremstyle{remark}
\newtheorem{remark}[theorem]{Remark}

\newtheoremstyle{exercise}% name
  {3pt}%      Space above
  {6pt}%      Space below
  {}%         Body font
  {}%         Indent amount (empty = no indent, \parindent = para indent)
  {\bfseries}% Thm head font
  {:}%        Punctuation after thm head
  { }%     Space after thm head: " " = normal interword space;
   {}%         Thm head spec (can be left empty, meaning `normal')
\theoremstyle{exercise}
\newtheorem{exercise}[theorem]{Exercise}
\newtheoremstyle{exercises}% name
  {3pt}%      Space above
  {6pt}%      Space below
  {}%         Body font
  {}%         Indent amount (empty = no indent, \parindent = para indent)
  {\bfseries}% Thm head font
  {:}%        Punctuation after thm head
  {\newline}%     Space after thm head: " " = normal interword space;
   {}%         Thm head spec (can be left empty, meaning `normal')

\theoremstyle{exercise}
\newtheorem{exercises}[theorem]{Exercises}

\input epsf
\def\boxit#1{\vbox{\hrule height1pt\hbox{\vrule width1pt\kern3pt
  \vbox{\kern3pt#1\kern3pt}\kern3pt\vrule width1pt}\hrule height1pt}}

\def\overarrow{\vec}

\def\bv{\bold v}

\def\BC{\mathbb C}\def\BN{\mathbb N}

\def\BP{\mathbb P}

\def\fgl{\mathfrak g\mathfrak l}

\def\tdim{{\rm dim}}

\def\hd{,...,}

\def\upperp{{}^\perp}

\def\cS{{\mathcal S}}

\def\11{\mathbf 1}

\def\fsl{{\mathfrak {sl}}}

\def\a{\alpha}

\def\s{\sigma}

\def\ot{{\mathord{ \otimes } }}
\def\op{{\mathord{\,\oplus }\,}}
\def\otc{{\mathord{\otimes\cdots\otimes}\;}}

\def\ra{{\mathord{\;\rightarrow\;}}}

\def\overarrow{\vec}
\def\frak{\mathfrak}
\def\fgl{\frak g\frak l}\def\fsl{\frak s\frak l}

\def\op{\oplus}

\def\fff#1#2{\Bbb F\Bbb F^{#1}_{#2}}

\def\ep{\epsilon}
\def\op{\oplus}

\def\s{\sigma}

\def\a{\alpha}

\def\ol{\overline}

\def\BP{\mathbb  P}
\def\BC{\mathbb  C}

\def\ep{\epsilon}

\def\opc{\op\cdots\op}

\def\eee#1#2#3{e^{#1}_{{#2}{#3}}}

\def\hd{, \hdots ,}

\def\ra{\rightarrow}

\def\tdet{\operatorname{det}}

\def\tperm{\operatorname{perm}}
\def\ttrace{\operatorname{trace}}

\def\tdim{\operatorname{dim}}

\def\tlim{\lim}

\def\upperp{{}^{\perp}}

\def\ctimes{\times \cdots\times}

\def\be{\begin{equation}}
\def\ene{\end{equation}}

\def\vp{{\bold V\bold P}}

\def\G{\Gamma}

\newcommand{\tEnd}{\operatorname{End}}

\def\bv{\bold v}\def\bV{\bold V}
\def\eee{\bold e}\def\fff{\bold f}

\begin{document}

\title{On the geometry of tensor network states}
\author{J.M. Landsberg, Yang Qi, and Ke Ye}
 
\begin{abstract} We answer a question of L. Grasedyck that arose in quantum information theory,
showing that the limit of tensors in a space of tensor network states need not be a tensor network
state. We also give geometric descriptions of 
spaces of tensor networks states corresponding
to trees and loops. Grasedyck's question has a surprising connection to the area of Geometric Complexity Theory, in that
the result is equivalent to the statement that the boundary of the Mulmuley-Sohoni type variety
associated to matrix multiplication is strictly larger than the projections  
of matrix multiplication (and re-expressions of matrix multiplication and its projections after changes of bases). Tensor Network States are also
related to graphical models in algebraic statistics.
\end{abstract}
\thanks{Landsberg  supported by NSF grant  DMS-1006353}
\email{jml@math.tamu.edu, yangqi@math.tamu.edu,kye@math.tamu.edu}
\maketitle

\section{Introduction}

\subsection{Origin in physics}
Tensors describe states of quantum mechanical systems. If a system has $n$ particles,
its state is an element of $H_1\otc H_n$ with $H_j$ Hilbert spaces.   In numerical many-body
physics, in particular solid state physics,   one wants to simulate
quantum states of thousands of particles, often arranged on a regular
lattice (e.g.,  atoms in a crystal). Due to the exponential growth of the 
dimension of $H_1\otc H_n$ with $n$, any na\"\i ve method of representing these
tensors  is
  intractable on a computer.  Tensor network states  were defined to
reduce the complexity of the spaces involved by restricting to a subset
of tensors that is physically reasonable, in the sense that
the corresponding spaces of tensors are only  locally  entangled
because 
 interactions (entanglement) in the physical world
 appear  to just happen locally.

  Such spaces have
been studied since the 1980's. These spaces are associated
to   graphs, and go under different names: {\it tensor network states},  {\it finitely correlated states (FCS)},
{\it valence-bond solids (VBS)},  {\it  matrix product
states (MPS)},  {\it  projected entangled pairs states (PEPS)},   and {\it  multi-scale entanglement renormalization ansatz
states (MERA)}, see, e.g., \cite{SandvikVidal,MR1158756,MR2284036,MR1104168,MR2231431,MR2566331} and the references therein. We will use the term tensor network states.

\subsection{Definitions and notation} 
For a graph $\G$ with edges $e_s$ and vertices $v_j$,
$s\in e(j)$ means   $e_s$ is  incident to $v_j$. If
$\G$ is directed, $s\in in(j)$ are the incoming edges
and $s\in out(j)$ the outgoing edges.

Let $V_1\hd V_n$ be complex vector spaces, let $\bv_i=\tdim V_i$.
Let $\G$ be a graph with $n$ vertices $v_j$, $1\leq j\leq n$, and
$m$ edges $e_s$, $1\leq s\leq m$,
and let $\overarrow{\eee}=(\eee_1\hd \eee_m)\in \BN^m$. Associate $V_j$ to the vertex $v_j$
and an auxiliary vector space $E_s$ of dimension $\eee_s$ to the
edge $e_s$. Make $\G$ into a directed graph. (The choice of directions will not effect the end result.)
Let $\bV=V_1\otc V_n$.

Let
\begin{align}
TNS(&\G,\overarrow{\eee}, \bV):=\\
&
\{ T\in \bold V\mid
\exists  T_j\in V_j\ot(\ot_{s\in in(j)}E_s) \ot (\ot_{t\in out(j)}E_t^*), 
\nonumber
{\rm \ such\ that\ } T=Con(T_1\otc T_n)
\}
\end{align}
where $Con$ is the  contraction of all the $E_s$'s with all
the $E_s^*$'s.

\begin{example} Let $\G$ be a graph with   two vertices and one
edge connecting them, then, 
$
TNS(\G,\eee_1, V_1\ot V_2)
$ 
is just the set of elements of $V_1\ot V_2$ of rank at
most $\eee_1$,   denoted $\hat\s_{\eee_1}(Seg(\BP V_1\times\BP V_2))$
and called the (cone over the) $\eee_1$-st secant variety 
of the Segre variety. To see this, let $\ep_1\hd \ep_{\eee_1}$ be
a basis of $E_1$ and $\ep^1\hd \ep^{\eee_1}$ the dual basis of $E^*$.
Assume, to avoid trivialities, that $\bv_1,\bv_2\geq \eee_1$.
Given $T_1\in V_1\ot E_1$ we may write
$T_1=u_1\ot \ep_1+\cdots + u_{\eee_1}\ot \ep_{\eee_1}$ for some
$u_{\a}\in V_1$. Similarly, given 
$T_2\in V_2\ot E_1^*$ we may write
$T_1=w_1\ot \ep^1+\cdots + w_{\eee_1}\ot \ep^{\eee_1}$ for some
$w_{\a}\in V_2$. Then $Con(T_1\ot T_2)=u_1\ot w_1
+\cdots + u_{\eee_1}\ot w_{\eee_1}$.
\end{example}

The graph used to define a set of tensor network states is often modeled to mimic the physical arrangement
of the particles, with edges connecting nearby particles, as nearby particles are the ones likely to be entangled.

\begin{remark} The construction of tensor network states in the physics literature does
not use a directed graph, because all vector spaces are Hilbert spaces, and thus
self-dual. However the sets of tensors themselves do not depend on
the Hilbert space structure of the vector space, which is why we omit this structure.
The small price to pay is the edges of the graph must be oriented, but all orientations
lead to the same set of tensor network states. 
\end{remark}

\subsection{Grasedyck's question} Lars Grasedyck asked:

\smallskip

{\it Is $TNS(\G,\overarrow{\eee}, \bV)$ Zariski closed? That is, given a sequence
of tensors $T_{\ep}\in \bV$ that converges to a tensor $T_0$, if $T_{\ep}\in TNS(\G,\overarrow{\eee}, \bV)$
for all $\ep\neq 0$,
can we conclude $T_{0}\in TNS(\G,\overarrow{\eee}, \bV)$?}

\smallskip

He mentioned that he could show this to be true when $\G$ was
a tree, but did not know the answer when  $\G$ is a triangle.

\begin{figure}[!htb]\begin{center}
\includegraphics[scale=.3]{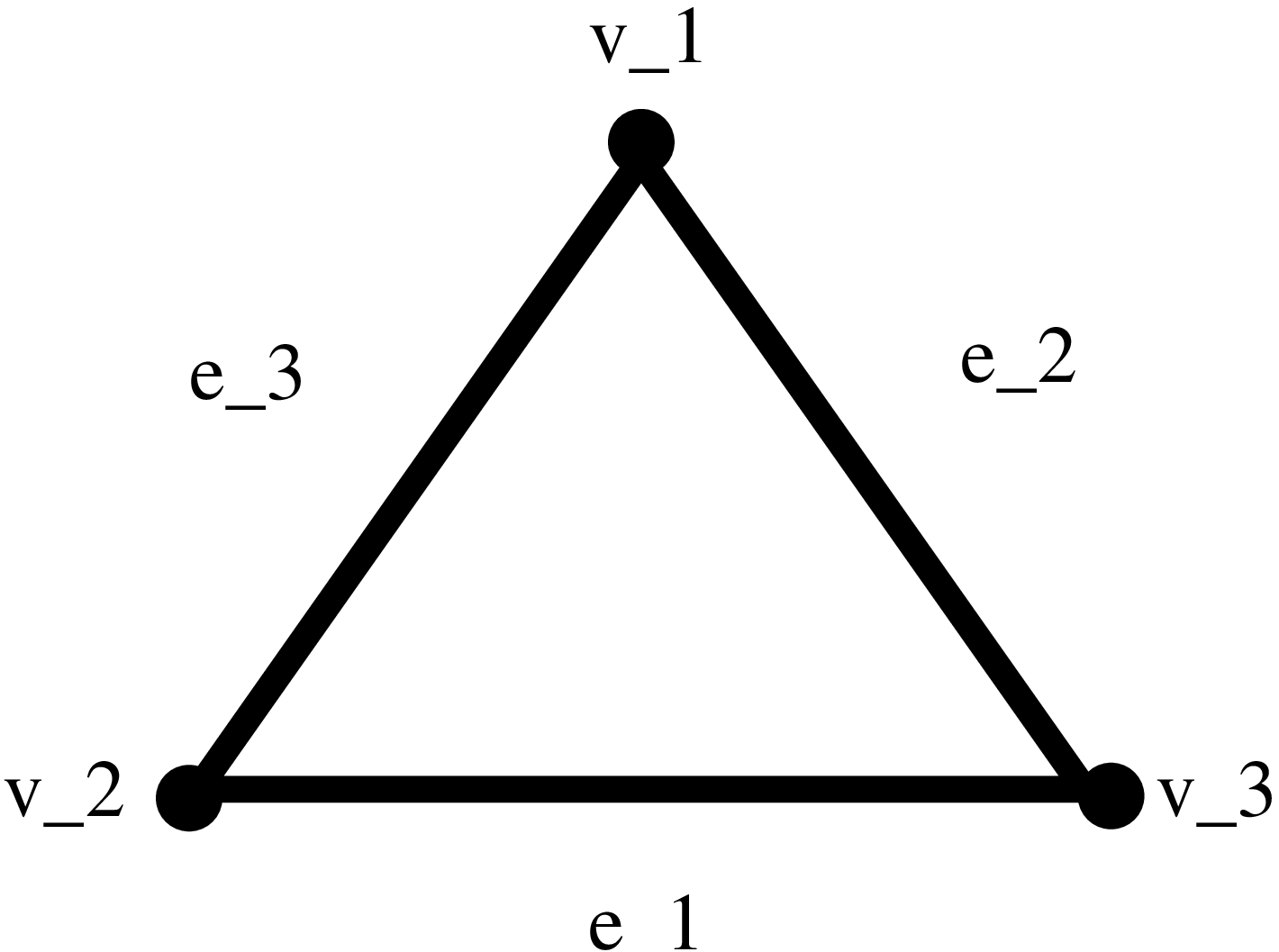}
%\caption{\small{ ... }}  
\end{center}
\end{figure}

\begin{definition} A dimension $\bv_j$ is {\it critical}, resp. {\it subcritical}, resp. 
{\it supercritical},
if $\bv_j=\Pi_{s\in e(j)}\eee_s$, resp. $\bv_j\leq \Pi_{s\in e(j)}\eee_s$, resp. $\bv_j\geq \Pi_{s\in e(j)}\eee_s$.
If $TNS(\G,\overarrow{\eee}, \bV)$ is critical for all $j$, we say $TNS(\G,\overarrow{\eee}, \bV)$
is critical, and similarly for sub- and super-critical.
\end{definition}

\begin{theorem} $TNS(\G,\overarrow{\eee}, \bV)$ is not Zariski closed for any $\G$ containing
a cycle whose vertices have  non-subcritical dimensions.  
\end{theorem}

\subsection*{Notation} $GL(V)$ denotes the group of  invertible linear maps $V\ra V$. $GL(V_1)\ctimes GL(V_n)$ acts on
$V_1\otc V_n$ by $(g_1\hd g_n)\cdot v_1\otc v_n=(g_1v_1)\otc (g_nv_n)$. (Here $v_j\in V_j$ and the action  on a tensor that is
a sum of rank one tensors is the sum of the actions on the rank one tensors.) Let $\tEnd(V)$ denote the set of all linear maps
$V\ra V$. We adopt the convention that $\tEnd(V_1)\ctimes \tEnd(V_n)$ acts on $V_1\otc V_n$  by
$(Z_1\hd Z_n)\cdot v_1\otc v_n=(Z_1v_1)\otc (Z_nv_n)$.   Let $\fgl(V)$ denote the Lie algebra of $GL(V)$. It is
naturally isomorphic to $\tEnd(V)$ but it acts on $V_1\otc V_n$ via the Leibnitz rule:
$(X_1\hd X_n)\cdot v_1\otc v_n=(X_1v_1)\ot v_2\otc v_n+ v_1\ot (X_2v_2)\ot v_3\otc v_n+\cdots v_1\otc v_{n-1}\ot (X_nv_n)$.
(This is because elements of the Lie algebra should be thought of as derivatives of curves in the Lie group at the identity.)
If $X\subset V$ is a subset, $\ol{X}\subset V$ denotes its closure. This closure is the same whether one uses
the Zariski closure, which is the common zero set of all polynomials vanishing on $X$, or the Euclidean closure,
where one fixes a metric compatible with the linear structure on $V$ and takes the closure with respect to limits.

\subsection{Connections to the GCT program}
The triangle case is especially interesting because we remark below that in the critical dimension case  it corresponds to
$$
End(V_1)\times End(V_2)\times End(V_3)\cdot Mmult_{\eee_3,\eee_2,\eee_1},
$$
where, setting $V_1=E_2^*\ot E_3$, $V_2=E_3^*\ot E_1$,  and
$V_3=E_2\ot E_1^*$,    $Mmult_{\eee_3,\eee_2,\eee_1}\in V_1\ot V_2\ot V_3$ is
the matrix multiplication operator, that is, as a tensor, 
$MMult_{\eee_3,\eee_2,\eee_1}=Id_{E_3}\ot Id_{E_2}\ot Id_{E_1}$. In
\cite{BIrank}    a {\it geometric complexity theory} (GCT) study
of $MMult$ and its  $GL(V_1)\times GL(V_2)\times
GL(V_3)$ orbit closure is considered. One sets $\eee_1=\eee_2=\eee_3=n$ and studies the
geometry as $n\ra \infty$.   It is a toy case of the varieties
introduced by Mulmuley and Sohoni \cite{MS1,MS2,BLMW},
  letting $S^d\BC^k$ denote the homogeneous polynomials of degree $d$ on $(\BC^k)^*$, 
the varieties are $\ol{GL_{n^2}\cdot  \tdet_n }\subset   S^n\BC^{n^2}$ and
$\ol{GL_{n^2}\cdot  \ell^{n-m}\tperm_m }\subset   S^n\BC^{n^2}$.
Here $\tdet_n\in S^n\BC^{n^2}$ is the determinant, a homogeneous polynomial
of degree $n$ in $n^2$ variables,  $n>m$,  $\ell\in S^1\BC^1$, $\tperm_m\in S^m\BC^{m^2}$ is the permanent
and an inclusion $\BC^{m^2+1}\subset \BC^{n^2}$ has been chosen.
In \cite{LMRdet} it was shown that
$End_{\BC^{n^2}}\cdot \tdet_n\neq \ol{GL_{n^2}\cdot \tdet_n}$, and   determining the  difference
between these sets is a subject of current research.

\smallskip

The critical loop case with $\eee_s=3$ for all $s$ is also related to the GCT program, as
it corresponds to the multiplication of $n$ matrices of size three. As a tensor, it may be
thought of as a map
$ (X_1\hd X_n)\mapsto\ttrace(X_1\cdots X_n)$.
This sequence of functions,  indexed by $n$, considered as a sequence of   homogeneous polynomials of degree $n$ on
$V_1\op\cdots \op V_n$, is complete for the class $\vp_e$ of sequences of polynomials of small formula
size, see \cite{clevebenor}.

\subsection*{Acknowledgments} We thank    David Gross for providing background information,
  Lars Grasedyck for introducing us to tensor network states and posing the Zariski closure question, and the referee for
suggestions that have substantially improved the presentation of the paper.

\section{Critical loops}

\begin{proposition}
  Let $\bv_1=\eee_2\eee_3,\bv_2=\eee_3\eee_1,\bv_3=\eee_2\eee_1$. Then
$
TNS(\triangle, (\eee_2\eee_3,\eee_3\eee_1,\eee_2\eee_1), V_1\ot V_2\ot V_3)
$
consists of matrix multiplication   and its degenerations (and their different expressions after changes of bases), i.e., 
$$
TNS(\triangle, (\eee_2\eee_3,\eee_3\eee_1,\eee_2\eee_1), V_1\ot V_2\ot V_3)= End(V_1)\times End(V_2)\times End(V_3)\cdot M_{\eee_2,\eee_3,\eee_1}.
$$
It has dimension $\eee_2^2\eee_3^2+ \eee_2^2\eee_1^2+\eee_3^2\eee_1^2 -
(\eee_2^2+\eee_3^2+\eee_1^2-1)$.

More generally, if $\G$ is a critical loop, 
$
TNS(\G, (\eee_n\eee_1,\eee_1\eee_2\hd \eee_{n-1}\eee_n), V_1\otc  V_n)
$
is $\tEnd(V_1)\ctimes \tEnd(V_n)\cdot M_{\overarrow {\bold e}}$, where
$M_{\overarrow {\bold e}}: V_1\ctimes V_n\ra \BC$ is the matrix multiplication operator
$(X_1\hd X_n)\mapsto \ttrace(X_1\cdots X_n)$.
\end{proposition}

\begin{proof} For the triangle case, a generic element
$T_1\in   E_2\ot E_3^*\ot  V_1$
may be thought of as a linear isomorphism
$E_2^*\ot E_3\ra V_1 $, identifying
$V_1$ as a space of $\eee_2\times \eee_3$-matrices,
and similarly for $V_2,V_3$.
Choosing bases $e^{u_s}_s$ for $E_s^*$, with dual basis $e_{u_s,s}$ for $E_s$,  
induces bases $x^{u_2}_{u_3}$ for $V_1$ etc.. Let $1\leq i\leq \eee_2$, $1\leq \a\leq \eee_3$,
$1\leq u\leq \eee_1$. Then
$$
con(T_1\ot T_2\ot T_3)=
\sum x^i_\a\ot y^\a_u\ot z^u_i
$$
which is the matrix multiplication operator. The general case is similar.
\end{proof}

\begin{proposition}\label{tristab} The Lie algebra of the stabilizer of $M_{\eee_n\eee_1,\eee_1\eee_2\hd \eee_{n-1}\eee_n}$ in
$GL(V_1) \ctimes GL(V_n)$ is the image of  
$ \fsl(E_1)\opc \fsl(E_n)$ under the map 
\begin{align*}
\a_1\opc \a_n\mapsto &(Id_{E_n}\ot \a_1, -\a_1^T\ot Id_{E_2},0\hd 0)
+(0,  Id_{E_1}\ot \a_2, -\a_2^T\ot Id_{E_3},0\hd 0)\\
&+\cdots + (-\a_n^T\ot Id_{E_1},0\hd 0,Id_{E_{n-1}}\ot \a_n).
\end{align*}
Here $\fsl(E_j)\subset \fgl(E_j)$ denotes the traceless endomorphisms and $T$ as a superscript denotes transpose  (which is
really just cosmetic).
\end{proposition}

The proof is safely left to the reader.

Large loops are referred to as \lq\lq 1-D systems
with periodic boundary conditions\rq\rq\  in the physics literature and are often  used in   simulations. By Proposition \ref{tristab},
for a critical loop, $\tdim(TNS(\G,\overarrow{\eee},\bold V))=\eee_1^2\eee_2^2+\cdots + \eee_{n-1}^2\eee_n^2+\eee_n^2\eee_1^2
-(\eee_1^2+\cdots +\eee_n^2-1)$, compared with the ambient space which has
dimension $\eee_1^2\cdots  \eee_n^2$. For example, when $\eee_j=2$ for all $j$,  
  $\tdim (TNS(\G,\overarrow{\eee},\bold V))=12n+1$, compared with $\tdim
\bold V= 4^n$.

\section{Zariski closure}

\begin{theorem}\label{tnsthm}  Let $\bv_1=\eee_2\eee_3,\bv_2=\eee_3\eee_1,\bv_3=\eee_2\eee_1$. Then
$
TNS(\triangle, (\eee_2\eee_3,\eee_3\eee_1,\eee_2\eee_1), V_1\ot V_2\ot V_3)
$
is not Zariski closed.
More generally any $TNS(\G,\eee,\bold V)$ where $\G$ contains a cycle with no subcritical vertex is
not Zariski closed.
\end{theorem}

\begin{proof}
Were $
T(\triangle):=TNS(\triangle, (\eee_2\eee_3,\eee_3\eee_1,\eee_2\eee_1), V_1\ot V_2\ot V_3)
$ Zariski closed, it  would be
\be\label{wereitclosed}
\ol{ GL(V_1)\times GL(V_2)\times GL(V_3)\cdot M_{\eee_2,\eee_3,\eee_1}}.
\ene
To see this, note that the $G=GL(V_1)\times GL(V_2)\times GL(V_3)$ orbit
of matrix multiplication is a Zariski open subset of $T(\triangle)$
of the same dimension as $T(\triangle)$.

We need to find a curve $g(t)=(g_1(t),g_2(t),g_3(t))$ such that $g_j(t)\in GL(V_j)$ for all $t\neq 0$ 
and $\tlim_{t\ra 0} g(t)\cdot  M_{\eee_2,\eee_3,\eee_1}$ is both defined and not in  $End(V_1)\times End(V_2)\times End(V_3)\cdot M_{\eee_2,\eee_3,\eee_1}$.

Note that for $(X,Y,Z)\in GL(V_1)\times GL(V_2)\times GL(V_3)$, we have  $(X,Y,Z) \cdot M_{\eee_2,\eee_3,\eee_1}(P,Q,R)
=\ttrace (X(P) Y(Q)Z(R))$. Here
$X: E_2^*\ot E_3\ra E_2^*\ot E_3$,
$Y: E_3^*\ot E_1\ra E_3^*\ot E_1$, $Z: E_1^*\ot E_2\ra E_1^*\ot E_2$.

 Take subspaces
$U_{E_2E_3}\subset E_2^*\ot E_3$, $U_{E_3E_1}\subset E_3^*\ot E_1$. Let $U_{E_1E_2}:=Con(U_{E_2E_3},U_{E_3E_1})\subset E_2^*\ot E_1$
be the images of all the $pq\in E_2^*\ot E_1$ where $p\in U_{E_2E_3}$ and $q\in U_{E_3E_1}$
(i.e.,  the   matrix multiplication of all pairs of elements).  Take $X_0,Y_0,Z_0$ respectively
to be the projections to 
$U_{E_2E_3}$, $U_{E_3E_1}$ and $ U_{E_1E_2}\upperp$. 
Let $X_1,Y_1,Z_1$ be the projections to   complementary
spaces (so, e.g.,  $X_0+X_1=Id_{V_1^*}$). For $P\in V_1^*$, write $P_0=X_0(P)$ and $P_1=X_1(P)$,  and similarly for $Q,R$.

Take  the curve $(X_t,Y_t,Z_t)$ with 
$X_t=\frac 1{\sqrt{t}}(X_0 + tX_1)$, $Y_t=\frac 1{\sqrt{t}}(Y_0+tY_1)$, $Z_t=\frac 1{\sqrt{t}}(Z_0+tZ_1)$.
 Then the limiting tensor,
as a map $V_1^*\times V_2^*\times V_3^* \ra\BC$, is 
$$(P,Q,R)\mapsto \ttrace(P_0Q_0R_1)+ \ttrace(P_0Q_1R_0)+\ttrace(P_1Q_0R_0).
$$
Call this tensor $\tilde M$. First observe that $\tilde M $ uses all the variables (i.e.,
considered   as a linear map  $\tilde M: V_1^*\ra V_{2}\ot V_3$, it is injective, and similarly for   its cyclic permutations). 
 Thus it is either in the orbit of  matrix multiplication
or a point in the boundary that is not in 
$End(V_1)\times End(V_2)\times End(V_3)\cdot M_{\eee_2,\eee_3,\eee_1}$, because all such boundary points  have at least one
such linear map non-injective.   

It remains to show that there exist $\tilde M$ such that  $\tilde M\not\in G\cdot M_{\eee_2,\eee_3,\eee_1}$
To prove some  $\tilde M$ is a point in the boundary, we compute the Lie algebra of its stabilizer and
show  it has dimension greater than 
  the the dimension of the stabilizer of matrix multiplication. 
One may take block matrices, e.g., 
$$
X_0=\begin{pmatrix} 0&*\\ *&0\end{pmatrix}, \ X_1=\begin{pmatrix} *& 0\\ 0&*\end{pmatrix},
$$
and $Y_0,Y_1$ have similar shape, but $Z_0,Z_1$ have the shapes reversed. Here one takes any splitting
$\eee_j=\eee_j'+\eee_j''$ to obtain the blocks.

For another  example, if one takes $\eee_j=\eee$ for all $j$,  $X_0$, $Y_0$, $Z_1$ to be the diagonal matrices and
and $X_1$, $Y_1$, $Z_0$ to be the matrices with zero on the diagonal, then one obtains
a stabilizer of dimension $4\eee^2-2\eee> 3\eee^2-1$. (This example  coincides with the previous one  when all $\eee_j=2$.)

To calculate the stabilizer  of $\tilde{M}$,
   first write down the tensor expression 
of $\tilde{M}\in V_1\otimes V_2\otimes V_3$ with respect to  fixed bases of $V_1$, $V_2$, $V_3$.
 Then   set an equation $(X,Y,Z).\tilde{M}=0$ where $X\in \fgl(V_1)$, $Y\in \fgl(V_2)$ and $Z\in \fgl(V_3)$ are unknowns.
Recall that here the action of $(X,Y,Z)$ on $\tilde{M}$ is the Lie algebra action, 
so we obtain a collection of linear equations.
Finally we solve this collection  of linear equations and count the dimension 
of the solution space. This dimension is the dimension of the stabilizer of $\tilde{M}$ in $GL(V_1)\times GL(V_2)\times GL(V_3)$.

To give an explicit example, let  $\eee_1={\eee}_2={\eee}_3={\eee}$
 and let $X_0=diag(x_1^1,...,x_{\eee}^{\eee})$,
 $Y_0=diag(y_1^1,...,y_{\eee}^{\eee})$, $Z_0=diag(z_1^1,...,z_{\eee}^{\eee})$,
 $X_1=(x^i_j)-X_0$, $Y_1=(y^i_j)-Y_0$, $Z_1=(z^i_j)-Z_0$. 
Then \[\tilde{M}=\sum_{i,j=1}^{{\eee}}(x^{i}_{j}y^{j}_{j}+x^{i}_{i}y^{i}_{j})z^{j}_{i}.\]
Let $X=\sum a^{(^{i}_{j})}_{(^{k}_{l})}X^{(^{k}_{l})}_{(^{i}_{j})}$ 
be an element of $\fgl(V_1)$, where $\{X^{(^{k}_{l})}_{(^{i}_{j})}\}$
 is a basis of $\fgl(V_1)$, and define $Y$ and $Z$ in the same pattern with 
coefficients $b^{(^{i}_{j})}_{(^{k}_{l})}$'s and $c^{(^{i}_{j})}_{(^{k}_{l})}$'s,
 respectively. Consider the equation $(X,Y,Z).T=0$ and we want to solve this equation 
for $a^{(^{i}_{j})}_{(^{k}_{l})}$'s, $b^{(^{i}_{j})}_{(^{k}_{l})}$'s and $c^{(^{i}_{j})}_{(^{k}_{l})}$'s. 
For these equations to hold, the coefficients of $z^{j}_{i}$'s must be zero. That is, for each pair $(j,i)$ of indices we have:
\[\sum_{k,l=1}^{\eee}a^{(^{i}_{j})}_{(^{k}_{l})}x^{k}_{l}y^{j}_{j}+
b^{(^{j}_{j})}_{(^{k}_{l})}x^{i}_{j}y_{k}^{l}+a^{(^{i}_{i})}_{(^{k}_{l})}x^{k}_{l}y^{i}_{j}+b^{(^{i}_{j})}_{(^{k}_{l})}x^{i}_{i}y^{k}_{l}
+c^{(^{l}_{k})}_{(^{j}_{i})}(x^{k}_{l}y^{l}_{l}+x^{k}_{k}y^{k}_{l})=0.\]
For these equations to hold, the coefficients of $y^{r}_{s}$'s must be zero. For example, if $s\ne j$, $r\ne s$ then we have:
\[b^{(^{j}_{j})}_{(^{r}_{s})}x^{i}_{j}+b^{(^{i}_{j})}_{(^{r}_{s})}x^{i}_{i}+c^{(^{s}_{r})}_{(^{j}_{i})}x^{r}_{r}=0\]
Now coefficients of $x$ terms must be zero, for instance, if $i\ne j$ and $i\ne r$, then we have:
\[b^{(^{j}_{j})}_{(^{r}_{s})}=0,\ \ \ 
 b^{(^{i}_{j})}_{(^{r}_{s})}=0, \ \ \ 
 c^{(^{s}_{r})}_{(^{j}_{i})}=0.\] 
If one  writes down and solves all such linear equations,  the dimension of the solution is $4\eee^{2}-2\eee$.

The same construction works for larger loops and cycles in larger graphs as it is essentially local - one just
takes all other curves the constant curve equal to  the identity.
\end{proof}

\begin{remark}
When $\eee_1=\eee_2=\eee_3=2$ we obtain a codimension one component of the boundary. 
In general, the dimension of the stabilizer  is much larger than the dimension of $G$,
so the orbit closures of these points do not give rise to codimension one components of the boundary.
It remains an interesting problem to find the codimension one components of the boundary.
\end{remark}

\section{Algebraic geometry perspective}
For readers familiar with  algebraic geometry, we recast the previous section in the language of algebraic geometry and put it
in a larger context. This section also serves to motivate the proof of the previous section.

To make the parallel
with the GCT program clearer, we describe the Zariski closure as the cone over 
  the (closure of) the image 
of the rational map (i.e., the \lq\lq closure\rq\rq\  of the map defined on a Zariski open subset)
\begin{align}
\BP End(V_1)\times \BP End(V_2)\times \BP End(V_3)
&\dashrightarrow
\BP (V_1\ot V_2\ot V_3)\\
\nonumber
([X],[Y],[Z])&\mapsto  (X,Y,Z)\cdot [M_{\eee_2,\eee_3,\eee_1}].
\end{align}
(Compare with the map $\psi$ in \cite[\S 7.2]{BLMW}.) A dashed arrow is used to indicate the map is
not everywhere defined.

The indeterminacy locus (that is, points  $([X],[Y],[Z])$ where the map is not defined),  consists of $([X],[Y],[Z])$ such that
for all triples of matrices $P,Q,R$, $\ttrace( X( P) Y( Q) Z(R))=0$.
In principle one can obtain \eqref{wereitclosed} as the image of a map from  a succession of blow-ups  
of $\BP End(V_1)\times \BP End(V_2)\times \BP End(V_3)$. (See, e.g., \cite[p. 81]{Harris}
for the definition of a blow-up)

One way to attain a point in the indeterminacy locus is to take  $([X_0],[Y_0],[Z_0])$ as described in the proof.   Taking a  
curve in $G$ that limits to this point may or may not give something new. In the proof we gave two  explicit choices that
do  give something new.

A more invariant way to discuss that  $\tilde M\not\in End(V_1)\times End(V_2)\times End(V_3)\cdot M_{\eee_2,\eee_3,\eee_1}$ is to
consider an  auxiliary
variety, called a {\it subspace variety}, 
$$
Sub_{\fff_1\hd\fff_n}(\bold V):=\{
T\in V_1\otc V_n \mid
\exists V_j'\subset V_j, \tdim V_j'=\fff_j,
{\rm \ and \ } T\in V_1'\otc V_n'\}, 
$$
and observe that if $T\in \times_j \tEnd(V_j)\cdot M_{\overarrow {\bold e}}$
and $T\notin \times_j GL(V_j)\cdot M_{\overarrow {\bold e}}$, then
$T\in Sub_{\fff_1 , \hd\fff_n}(\bold V)$
where $\fff_j <\eee_j$ for at least one $j$.

The statement that \lq\lq $\tilde M$ uses all the variables\rq\rq\  may be rephrased as saying  that $\tilde M\notin Sub_{\eee_2\eee_3 -1,
\eee_2\eee_1-1,\eee_3\eee_1-1}(V_1\ot V_2\ot V_3)$

\section{Reduction from the supercritical case to the critical case
with the same graph}

For a vector space $W$, let $G(k,W)$ denote the Grassmannian of $k$-planes through the
origin in $W$. Let $\cS\ra G(k,W)$ denote the tautological rank $k$ vector bundle
whose fiber over $E\in G(k,W)$ is the $k$-plane $E$.
Assume $\fff_j\leq \bv_j$ for all $j$ with at least
one inequality strict. 
Form the vector bundle $\cS_1\otc \cS_n$ over $G(\fff_1,V_1)\ctimes G(\fff_n, V_n)$,
where $\cS_j\ra G(\fff_j,V_j)$ are the tautological
subspace bundles. 
Note that the total space of  $\cS_1\otc \cS_n$ maps to $\bold V$ with image
$Sub_{\overarrow\fff}(\bold V)$.
Define a   fiber sub-bundle, whose
fiber over $( U_1 \ctimes U_n)\in G(\fff_1,V_1)\ctimes G(\fff_n, V_n)$ is
$TNS(\G,\overarrow\eee, U_1\otc U_n)$. Denote this bundle by 
$TNS(\G,\overarrow\eee, \cS_1\otc \cS_n)$.

The supercritical cases may be realized, in the language of Kempf,  as a \lq\lq collapsing of a bundle\rq\rq\  over the critical
cases as follows:

\begin{proposition}
Assume $\fff_j:=\Pi_{s\in e(j)}\eee_s\leq \bv_j$.
Then $TNS(\G,\overarrow\eee, \bold V)$ is the image of the  bundle
$TNS(\G,\overarrow\eee, \cS_1\otc \cS_n)$ under the map to $\bold V$.
In particular
$$
\tdim( TNS(\G,\overarrow\eee, \bold V))=
\tdim( TNS(\G,\overarrow\eee, \BC^{\fff_1}\otc  \BC^{\fff_n}))
+   \sum_{j=1}^n\fff_j(\bv_j-\fff_j).
$$
\end{proposition}

\begin{proof} If $\Pi_{s\in e(j)}\eee_s\leq \bv_j$, then  any tensor
 $T\in   V_j\ot(\ot_{s\in in(j)}E_s) \ot (\ot_{t\in out(j)}E_t^*)$,
must lie in some $V_j'\ot(\ot_{s\in in(j)}E_s) \ot (\ot_{t\in out(j)}E_t^*) $
with $\tdim V_j'=\fff_j$. 
 The space $TNS(\G,\overarrow\eee, \bold V)$ is
the image of this subbundle under the map to $\bold V$.
\end{proof}

This type of  bundle
construction   is standard, see \cite{Ltensor,weyman}.
Using the techniques in \cite{weyman}, one may reduce questions about a supercritical case
to the corresponding critical case.

\section{Reduction of cases with subcritical vertices of valence one}
The subcritical case in general can be understood in terms of projections of critical
cases, but this is not useful for extracting information. However,  if a subcritical
vertex has valence one, one may simply reduce to a smaller graph as we now describe.

%\begin{figure}[!htb]\begin{center}
%\includegraphics[scale=.3]{tnetchain.eps}
%\caption{\small{ ... }}  
%\end{center}
%\end{figure}

\begin{figure}[!htb]\begin{center}
\includegraphics[scale=.4]{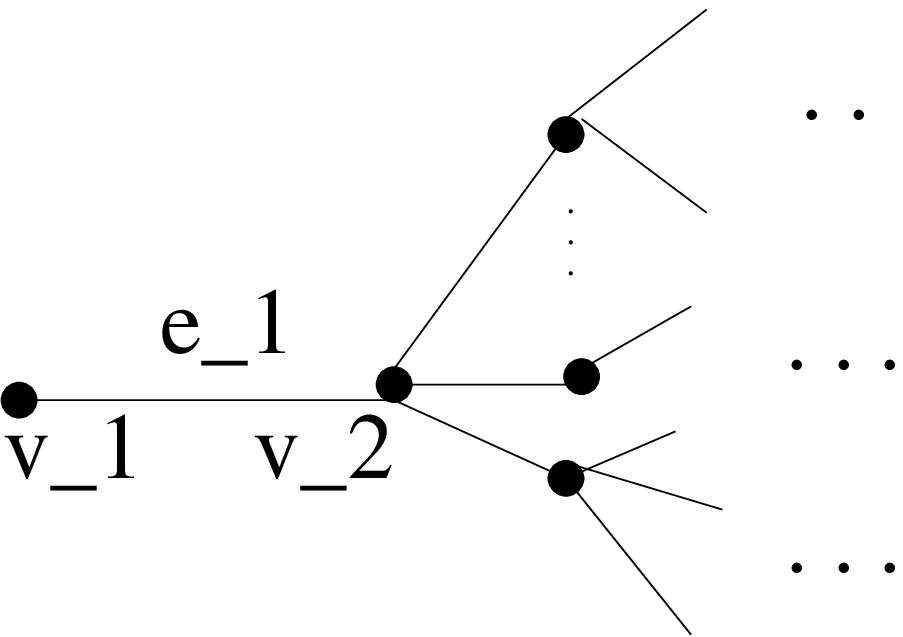}
%\caption{\small{ ... }}  
\end{center}
\end{figure}

\begin{proposition}\label{subcrred} Let
$TNS(\G,\overrightarrow{\eee},\bold V)$ be a tensor network
state, let $v$ be a vertex of  $\G$ with valence one.
Relabel the vertices such that $v=v_1$ and  so that
$v_1$ is attached by $e_1$ to $v_2$. 
If  $\bv_{1}\leq \eee_{1}$, then $TNS(\G,\overrightarrow{\eee},V_1\otc V_n )=TNS(\tilde\G ,\overrightarrow{\tilde\eee},\tilde V_{1}\otimes V_{3}\otimes...\otimes V_{n})$, where $\tilde \G$ is $\G$ with $v_1$ and $e_1$ removed, $\overrightarrow{\tilde \eee}$ is the vector $(\eee_{2},...,\eee_{n})$ and $\tilde V_{1}=V_{1}\otimes V_{2}$.  
\end{proposition}
\begin{proof}
 A general element in $TNS(\G,\overrightarrow{\eee},V_1\otc V_n)$ is of the form $\sum_{i,j=1}^{\eee_{1},\eee_{2}}$ $u_{i}\otimes v_{iz}\otimes w_{z}$,
where $w_z\in V_3\otc V_n$. Obviously, $TNS(\G,\overrightarrow{\eee},V_1\otc V_n)\subseteq TNS(\tilde\G ,\overrightarrow{\tilde\eee},\tilde V_{1}\otimes V_{3}\otimes...\otimes V_{n})=:TNS(\tilde\G ,\overrightarrow{\tilde\eee},\tilde{\bold V})$. Conversely, a general element in $TNS(\tilde\G ,\overrightarrow{\tilde\eee},\tilde{\bold V}))$ is of the form $\sum_{z} X_{z}\otimes w_{z}$, $X_{z}\in V_{1}\otimes V_{2}$. Since $\bv_{1}\leq \eee_{1}$, we may  express $X_{z}$ in the form $\sum_{i=1}^{e_{1}}u_{i}\otimes v_{iz}$, where   $u_{1},...,u_{v_{1}}$ is a basis of $V_{1}$. Therefore, 
$TNS(\G,\overrightarrow{\eee},\bold V)\supseteq 
TNS(\tilde\G ,\overrightarrow{\tilde\eee},\tilde{\bold V})$..
\end{proof}
 
\section{Trees}
With trees one can apply the two reductions successively to   reduce to a tower of bundles
where the fiber in the last bundle is a linear space. 
The point is that a critical vertex is both sub- and supercritical, so
one can reduce at valence one vertices iteratively. Here are a few examples in the special case of chains. The result is similar to the
Allman-Rhodes reduction theorem for phylogenetic trees \cite{AR1}.

\begin{example}
Let $\G$ be a chain with $3$ vertices. If it is supercritical, $TNS(\G,\overrightarrow{\eee},\bold V)=V_1\ot V_2\ot V_3$. Otherwise
$TNS(\G,\overrightarrow{\eee},\bold V)=  Sub_{\eee_1,\eee_1\eee_2,\eee_2}(V_1\ot V_2\ot V_3)$.
\end{example}

\begin{example}
Let $\G$ be a chain with $4$ vertices. If $\bv_{1}\leq \eee_{1}$ and $\bv_{4} \leq \eee_{3}$, then, writing $W=V_1\ot V_2$
and $U=V_3\ot V_4$, by
Proposition  \ref{subcrred}, $TNS(\G,\overrightarrow{\eee},\bold V)$ is the set of rank
at most $\eee_2$ elements in
$W\ot U$ (the secant variety of the two-factor Segre).
Other chains of length four have similar complete descriptions. 
\end{example}
\begin{example}
Let $\G$ be a chain with $5$ vertices. Assume that $\bv_{1}\leq \eee_{1}$,  $\bv_{5}\leq \eee_{4}$ and
$\bv_1\bv_2\geq \eee_2$ and $\bv_4\bv_5\geq \eee_3$. Then $TNS(\G,\overrightarrow{\eee},\bold V)$ is the image of a bundle
over $G(\eee_2,V_1\ot V_2)\times G(\eee_3,V_4\ot V_5)$ whose fiber is the set of tensor network states
associated to a chain of length three. 
\end{example}

\bibliographystyle{amsplain}
 
\bibliography{Lmatrix}

\end{document}